\documentclass{article}
\usepackage[utf8]{inputenc}
\usepackage{times}
\usepackage{authblk}
\usepackage{graphicx} 
\usepackage{amsmath}
\usepackage{amssymb}
\usepackage{amsthm}
\usepackage{cite}
\usepackage{epstopdf}
\usepackage{chemfig}
\usepackage{centernot} 
\usepackage{dirtytalk}

\theoremstyle{plain}
\usepackage{bbm}
\usepackage{pstricks}
\usepackage{tikz}
\usepackage{caption}
\usepackage{subcaption}

\newtheorem{theorem}{Theorem}

\newtheorem{definition}{Definition}
\newtheorem{proposition}{Proposition}
\newtheorem{example}{Example}

\providecommand{\keywords}[1]
{
  \small	
  \textbf{\textit{Keywords---}} #1
}
\title{A Note On $\ell$-Rauzy Graphs for the Infinite Fibonacci Word}
\author{Rajavel Praveen M and Rama R \\ Department of Mathematics\\ Indian Institute of Technology Madras, Chennai - 600036, India \\ kingspearpraveen@gmail.com, ramar@iitm.ac.in }
\date{}

\begin{document}

\maketitle
\begin{abstract}
    The $\ell$-Rauzy graph of order $k$ for any infinite word is a directed graph in which an arc $(v_1,v_2)$ is formed if the concatenation of the word $v_1$ and the suffix of $v_2$ of length $k-\ell$ is a subword of the infinite word. In this paper, we consider one of the important aperiodic recurrent words, the infinite Fibonacci word for discussion. We prove a few basic properties of the $\ell$-Rauzy graph of the infinite Fibonacci word. We also prove that the $\ell$-Rauzy graphs for the infinite Fibonacci word are strongly connected.
\end{abstract}
\section{Introduction}
Theory of word representable graphs have main applications in Combinatorics, Graph theory, Computer science and Algebra \cite{beigel,cerney,lozin}. This theory was first introduced by S. V. Kitaev and studied in detail \cite{kit} by the motivation of Perkins semigroup in \cite{seif}. A detailed survey is made by S. V. Kitaev and A. V. Pyatkin in \cite{kitpyat}. Fundamental properties of word representable graphs are discussed in \cite{kit,kitsali}. 

A de Bruijn graph of order $m$, is a directed graph with the vertex set $\Sigma ^k$ and $uv$ forms an arc iff $u_2u_3\cdots u_k=v_1v_2\cdots v_{k-1}$. There are many interesting results like \say{For a de Bruijn graph of order $k$ whose $|\Sigma|=2$ with $2^k$ vertices, there exists $2^{2^{k-1}-k}$ different Hamiltonian cycles} \cite{debruijn}. Finding Hamiltonian cycle in a graph is a difficult computational problem, where de Bruijn graph made it easier and it is widely applied in genome assembly \cite{philip}.

In \cite{rauzy}, Gerard Rauzy introduced a new graph by adding some more conditions to de Bruijn graph, called as Rauzy graph. The graph is defined with vertex set $\Sigma ^k\cap L (\text{denote} \Sigma ^k\cap L \text{ by }L(k))$ and an ordered pair $(u,v)$ of vertices form an arc iff $u_2u_3\cdots u_k=v_1v_2\cdots v_{k-1}$ and $uv_k \in L(k+1)$. Here, $L$ denotes the factorial language, and $\Sigma^k$ is the set of all words of length $k$ from the alphabet $\Sigma$. Rauzy graphs are broadly used in finding the complexity of words of finite lengths. Arnoux and Rauzy in 1991, investigated the sequences with complexity $2n+1$. And G. Rote in \cite{rote}, went one step further to Arnoux and Rauzy by constructing the sequences with complexity $2n$ using Rauzy graphs. Then Ali Aberkane in \cite{aliaberkane}, approached similarly the intermediate case of complexity between $n+1$ and $2n$. 

In \cite{frid}, Frid  obtained a description of Rauzy graphs for a wide family of sequences. The author proved that to find the structure of Rauzy graphs for arbitrarily long lengths, it is sufficient to find a fixed number of Rauzy graphs for length bounded by a constant. In \cite{ali}, Ali Aberkane studies the infinite words whose $lim \frac{p(n)}{n}=1$, with the help of Rauzy graphs. Salimov in \cite{salimov}, proved that for a given sequence of strongly connected graphs with maximal in and out degrees equal to $s$, an uniformly recurrent infinite word on $\Sigma$, $|\Sigma| =s$ can be constructed. In the sequence of its Rauzy graphs, there is a subsequence of graphs isomorphic to the stretchings of graphs of the given sequence. In \cite{bal}, Balkov\'a et al. proves that the factor frequency of infinite words whose language is closed under reversal does not exceed $2\Delta C(n)+1$.

Later in \cite{krp}, we introduced a variant of Rauzy graph in which the vertex set is same as the Rauzy graph but any two vertices $(u,v)$ form an arc iff $u[\frac{|u|}{2}+1,|u|]=v[1,\frac{|v|}{2}]$ i.e, half the length of the vertices are matched instead of $|u|-1$ to form an arc. The idea of sharing half the length of vertex was motivated by the encoding procedure of vertices and edges in to DNA strand, proposed by Adleman in \cite{adleman}. Some interesting structural properties of half range Rauzy graphs were studied in \cite{krp}.

In this paper, we generalize the sharing length $\ell$ of suffix/prefix in vertices to form an arc i.e., $(u,v)$ forms an arc iff $u[|u|-\ell+1,|u|]=v[1,\ell]$ and call it as $\ell-$ Rauzy graph. This is the generalization of Rauzy graphs and half range Rauzy graphs. 

In \cite{rote}, G. Rote proved that Rauzy graphs of any recurrent word are strongly connected. But, the $\ell$-Rauzy graph of any recurrent word need not be connected. For example, the $1$-Rauzy graph of order $4$ for an infinite periodic word $xxx\ldots$ is not connected, where $x$ is a primitive word with alphabet size $2$ and the length of $x$ is atleast 4. Also, the $2$-Rauzy graph of order $4$ for the Thue-Morse word (an aperiodic recurrent infinite word) is not connected. We observe that the $\ell$-Rauzy graphs of the infinite Fibonacci word are strongly connected. So, we are interested in proving that the $\ell$-Rauzy graph of any order $k(\in N)$ for the well known infinite Fibonacci word is strongly connected.

Fibonacci word is one of the most studied infinite word in combinatorics on words as it has many combinatorial properties. Fibonacci words are defined by one of the simplest morphisms $\phi: 0\to 01$, and $1\to 0$. Fibonacci word is a Sturmian word whose subword complexity, $\sigma(k)=k+1$. The subword complexity of the Fibonacci word is minimum among all aperiodic recurrent words. Fibonacci words are used to prove optimality of several results such as text algorithms and periodicity of infinite words. The finite Fibonacci words are considered as important as the Fibonacci numbers because of their applications. 

In \cite{chuan2}, Chuan uses Zeckendorf representation to obtain the locations of those subwords whose lengths are Fibonacci numbers $\geq 2$. Later in \cite{chuan}, Chuan obtain the locations of any finite subword of the Infinite Fibonacci word. In \cite{rytter}, Rytter also obtains the location of any finite subword of the Infinite Fibonacci word in a different approach. The locations of any finite subword of the Infinite Fibonacci word can also be known by using the software Walnut. For more details about the Walnut software, one may refer \cite{shallit}. Locations of the subwords plays a vital role in proving that the $\ell$-Rauzy graph of order $k$ for the infinite Fibonacci word is strongly connected for any $k$ and $1\leq \ell \leq k-1$.

\section{Preliminaries}
  In this section, we present few basic and necessary definitions, for more details one can refer\cite{lot,berthe,gary,bondy}. A non empty collection of symbols is an alphabet $\Sigma$. A sequence of finite or infinite symbols from $\Sigma $ forms a word. Length of a word $w$ is the number of letters in $w$, denoted by $l(w)$ and $\Sigma ^*$ is the set of all finite words and $\Sigma ^n$ is the set of all words over $\Sigma $ of length $n$. A word $u$ is a factor of $w=w_1w_2\ldots$, if $u=w_iw_{i+1}\ldots w_{i+k-1}$, and is denoted by $w[i;k]$ for some $i,k\in \mathbbmss{N}$. Here, $w_i$ denotes the symbol in the $i$th position of $w$, and $w[i;k]$ is the word that starts at position $i$ and has length $k$. Any factor $u$ is a prefix(suffix) of $w$ if $w=ux(w=xu)$, $x\in \Sigma^*$.
 
 A set $L \subseteq \Sigma ^*$ is said to be a factorial language if it contains all the subwords of its words. Let $L(k)=L\cap \Sigma ^k$, $L_w$ be the set of all factors of $w$ and $L_w(k)$ be the set of all factors of $w$ of length $k$.
 
  Let $g_n$ be the $n$th Fibonacci word, where $$ g_0=1,~g_1=0,~g_n=g_{n-1}g_{n-2},~n\geq 2. $$ The words $g_n$ are referred to as the finite Fibonacci words. Let $F_n$ be the $n$th Fibonacci number, where $|g_n|=F_n$. The limit $f=\lim\limits_{n\to \infty} g_n$ is called the infinite Fibonacci word. The infinite Fibonacci word is given by $$f=010010100100101001010\ldots$$ whose $n$th letter is $1~(resp.,~0)$ if $\lfloor (n+1)\tau \rfloor -\lfloor n\tau \rfloor=0~(resp.,~1),$ where $\tau=\frac{\sqrt{5}-1}{2}, n\geq 1 $ and the complement of infinite Fibonacci word is $f^c=101101011011010110101\ldots$.
  
A directed graph $G$ is an ordered pair $(V(G),E(G))$ consisting of non empty set $V(G)$ of vertices, a set $E(G)$, disjoint from $V(G)$, of arcs. In a graph $G$, indegree (resp., outdegree) of a vertex $u$ is the number of arcs entering (resp., leaving) $u$ and denoted by $deg_{in}(u)$ (resp., $deg_{out}(u)$). A vertex $u$ is isolated iff $deg_{in}(u)=0=deg_{out}(u)$. 

A directed graph is said to be connected (weakly) if there is a path between any two vertices in its underlying undirected graph. A directed graph is said to be strongly connected if it has a path from each vertex to every other vertex. A loop (or self-loop) is an edge from a vertex to itself. Simple directed graphs are directed graphs that have no loops and no multiple arcs.
\begin{definition}
A de Bruijn graph of order $k>1$ is a directed graph whose vertex set is $\Sigma ^k$ and an arc $uv$ is formed iff
\[ u[2,k]=v[1,k-1] \] 
\end{definition}
Some more conditions on de Bruijn graph were imposed by Rauzy and defined a graph in the following way:
 \begin{definition}
  A Rauzy graph of order $k$ for a factorial language $L$ is a directed graph $(V,E)$ where $V=L(k)$ and $(u,v)\in E$ iff
\[ u_2u_3\cdots u_{k}=v_1v_2\cdots v_{k-1} \qquad and \qquad u_1u_2\cdots u_{k}v_{k}\in L(k+1).\]
\end{definition}
 A Rauzy graph of order $k$ for an infinite word $w$ is the Rauzy graph of order $k$ for the language of subwords of $w$. We denote a Rauzy graph of order $k$ for a factorial language $L$ (for an infinite word $w$) by $R_L(k)$ (correspondingly, $R_w(k)$).
 
 Later, a new graph is defined from Rauzy graph by sharing the suffix of preceding vertex with the prefix of succeeding vertex by half the length of its vertices \cite{krp}.
 \begin{definition}
An \lq Half range Rauzy graph\rq (or and HRR-graph in short) of order $k>1$, for a factorial language $L$ is a directed graph $(V,E)$, where $V=L(k)$ and arc set is defined as follows:
\begin{enumerate}
    \item For an even $k$, $(u,v)\in E$ iff

$u_{\frac{k}{2}+1}u_{\frac{k}{2}+2}\cdots u_k=v_1v_2\cdots v_{\frac{k}{2}}~and~u_1u_2 \cdots u_kv_{\frac{k}{2}+1}\cdots v_k \in L(\frac{3k}{2}).$
\item For an odd $k$, there are two types of graphs, $(u,v) \in E$ iff

Type I: $u_{\frac{k+1}{2}}\cdots u_k = v_1\cdots v_{\frac{k+1}{2}}~~and~~
 u_1u_2 \cdots u_kv_{\frac{k+3}{2}} \cdots v_k \in L(\frac{3k-1}{2}) $
 
Type II: $u_{\frac{k+3}{2}}\cdots u_k=v_1\cdots v_{\frac{k-1}{2}} ~~ and ~~ u_1u_2\cdots u_kv_{\frac{k+1}{2}} \cdots v_k \in L(\frac{3k+1}{2}) $
\end{enumerate}
denoted by, $\mathbbmss{HR}_L(k,*)=\begin{cases}
\mathbbmss{HR}_L(k) & if ~k~ is~ even\\
\mathbbmss{HR}_L(k,I) & if~ k~ is~ odd~and~Type~I\\
\mathbbmss{HR}_L(k,II) & if~ k~ is~ odd~and~Type~II
\end{cases}$
 \end{definition}
 
 If the underlying language is set of all factors of a given word $w$, then $\mathbbmss{HR}_L(k,*)$ is simply represented as $\mathbbmss{HR}_w(k,*)$.

\section{The $\ell$-Rauzy graph}
Though, we were motivated by Adleman in \cite{adleman}, by matching half the length of DNA strands, the sharing length of suffix and prefix among the vertices made a difference in Rauzy graph and HRR (which is shown in \cite{krp}). Now, we are interested in the question \say{what if, we match an arbitrary length $1\leq \ell \leq k-1$ of suffix/prefix word among the vertices in a graph to form an arc?} On answering this question, a new graph $\ell-$Rauzy graph is defined as follows and its properties are studied.
\begin{definition}
  An $\ell$-Rauzy graph of order $k$ for a factorial language $L$ is a directed graph $(V,E)$ where $V=L(k)$ and any two vertices $u,v$ forms an edge i.e. $(u,v)\in E$ iff
\[ u_{k-\ell+1}u_{k-\ell+2}\cdots u_{k}=v_1v_2\cdots v_{\ell} ~ and ~ u_1u_2\cdots u_{k}v_{\ell+1}v_{\ell+2}\cdots v_{k} \in L(2k-\ell)\]
is denoted by $\ell$-$\mathbbmss{R}_L(k)$.
\end{definition}

 An $\ell$-Rauzy graph of order $k$ for an infinite word $w$ is the $\ell-$Rauzy graph of order $k$ for the language of subwords of $w$ and denoted by $\ell$-$\mathbbmss{R}_w(k)$.
 
\begin{example}
The $\ell$-Rauzy graphs of order $4$ for the word $w=010010010\ldots $ are directed graphs with vertex set $V_1=\{v_1=0100, ~v_2=1001,~ v_3=0010 \}$, and the arc set varies for various $\ell$. The graph of $\ell$-$R_w(4)$ is shown in Figure \ref{fig:1}.
\begin{figure}
    \centering
    \includegraphics[scale=.25]{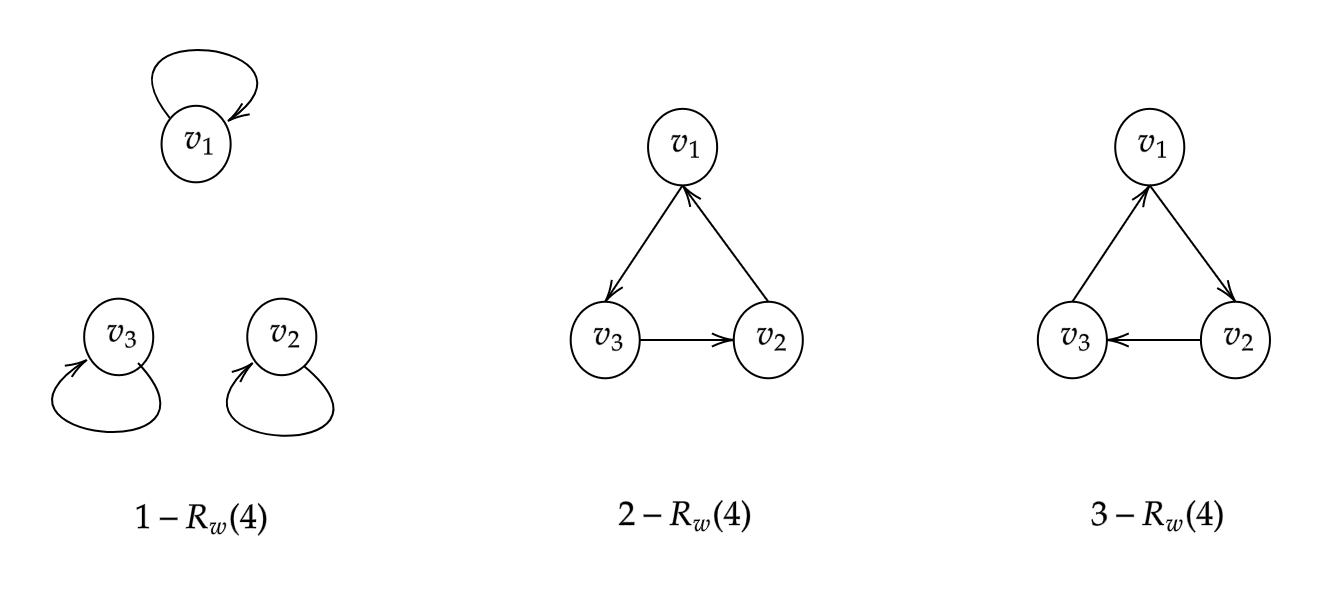}
    \caption{$\ell$-Rauzy graphs of order 4 for the word $w$}
    \label{fig:1}
\end{figure}
\end{example}
\begin{example}
$\ell$-Rauzy graphs  of order $4$ for the infinite Fibonacci word $f$ are directed graphs with vertex set $V_2=\{u_1=0100, ~u_2=1001,~ u_3=0010,~u_4=1010 \}$. For various $\ell$, graphs of $\ell$-$R_w(4)$ are shown in Figure \ref{fig:2}.
\begin{figure}
    \centering
    \includegraphics[scale=.25]{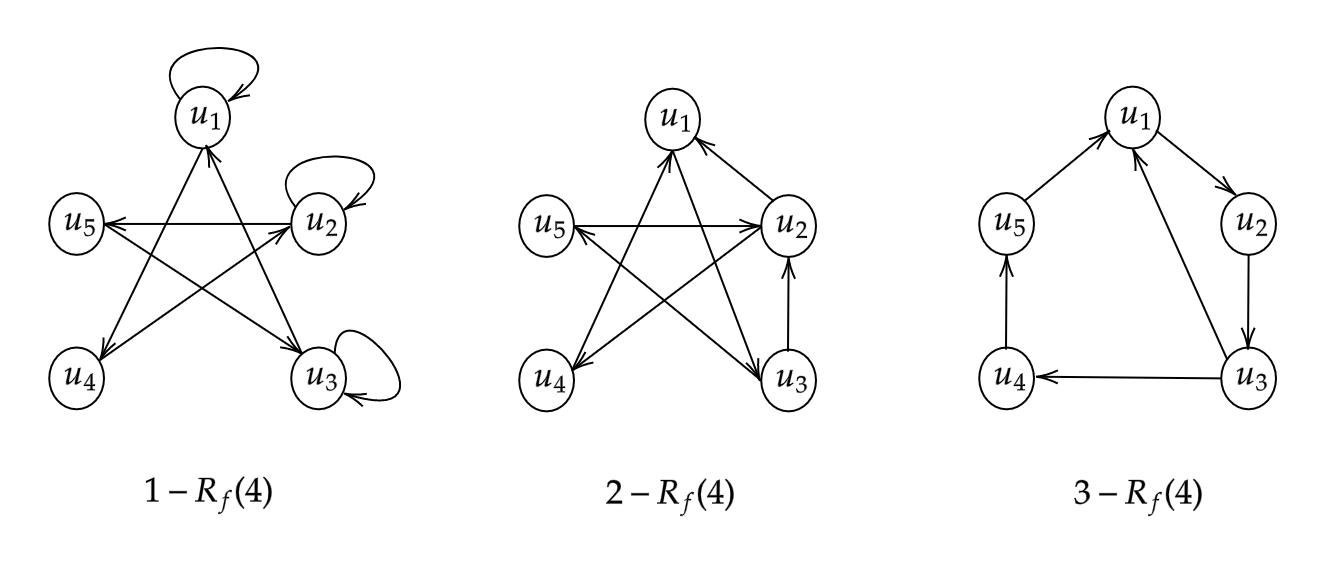}
    \caption{$\ell$-Rauzy graphs of order 4 for the infinite Fibonacci word $f$}
    \label{fig:2}
\end{figure}
\end{example}
\begin{example}
The $2$- Rauzy graph of order $4$ for Thue-Morse infinite word $T$ is a directed graph with vertex set $V_3=\{ v'_1=0110,~v'_2=1101,~v'_3=1010,~v'_4=0100,~v'_5=1001,~v'_6=0011,~v'_7=1100,~v'_8=0010,~v'_9=0101,~v'_{10}=1011 \}$. The graph $2$-$R_T(4)$ is shown in Figure \ref{fig:3}.
\begin{figure}
    \centering
    \includegraphics[scale=0.25]{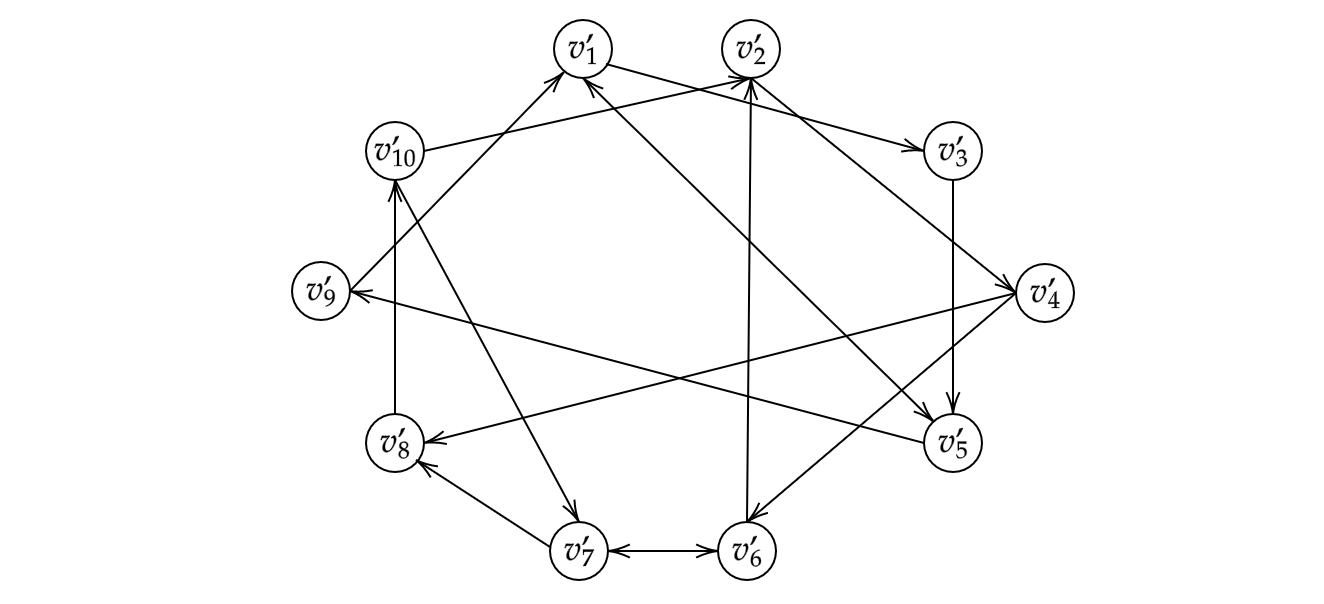}
    \caption{$2$-Rauzy graph of order 4 for the Thue-Morse infinite word $T$}
    \label{fig:3}
\end{figure}
\end{example}

\section{Properties of $\ell$-$\mathbbmss{R}_f(k)$ for the Infinite Fibonacci word}\label{sec4}

In this section, we discuss a few basic properties of $\ell$-Rauzy graph for the infinite Fibonacci word.

By definition of $\ell$-Rauzy graph of order $k$, the set of vertices is the set of all subwords of length $k$ in the factorial language $L_f(k)$ of infinite Fibonacci word $f$. The subword complexity of fibonacci infinite word is well known and there are $n+1$ number of subwords of length $n$. Therefore, the number of vertices in $\ell$-$\mathbbmss{R}_f(k)$ is given by
  $$ |V(\ell\text{-}\mathbbmss{R}_w(k))|= k+1.$$
  By definition of $\ell$-Rauzy graph for order $k$, the set of arcs is the set of all subwords of length $2k-\ell$  in the factorial language $L_f(2k-\ell)$ of infinite Fibonacci word $f$. As there are $2k-\ell+1$ subwords of length $2k-\ell$, the number of arcs in $\ell$-$\mathbbmss{R}_f(k)$ is given by

 $$|E(\ell\text{-}\mathbbmss{R}_f(k))|=2k-\ell+1.$$
 The following proposition ensures that none of the vertices of $\ell$-Rauzy graph for the infinite Fibonacci word is isolated.
   \begin{proposition}
     For each vertex $v$ in $\ell$-Rauzy graph for the infinite Fibonacci word, $deg_{in}(v)\geq 1$ and $deg_{out}(v)\geq 1$.
   \end{proposition}
   \begin{proof}
Let $v$ be a word $x_ix_{i+1}\cdots x_{i+k-1}$ of length $k$. As the infinite Fibonacci word is recurrent, there exist a $$u=\begin{cases}x_{i-k+\ell}x_{i-k+\ell+1}\cdots x_{i-1}x_ix_{i+1}\cdots x_{i+\ell-1}& \text{for $i>k-\ell$}\\
x_{j-k+\ell}x_{j-k+\ell+1}\cdots x_{j-1}x_jx_{j+1}\cdots x_{j+\ell-1}&\text{for $i\leq k-\ell$}
\end{cases}$$ where $x_ix_{i+1}\cdots x_{i+k-1}=x_jx_{j+1}\cdots x_{j+k-1}$, for some $j>i+k-\ell $ and $$u'=x_{i+\ell}\cdots x_{i+k-1}x_{i+k}x_{i+k+1}\cdots x_{i+k+\ell-1}$$ such that $(u,v),(v,u')\in E(\ell$-$\mathbbmss{R}_w(k))$. Hence $deg_{in}(v)\geq 1$ and $deg_{out}(v)\geq 1$.
   \end{proof}
   For given $k$ and $\ell$, the indegree and outdegree of any vertex in $\ell$-$R_f(k)$ can be known explicitly. Let $k=F_{n+1}-1$ and $F_{n-1}\leq k-\ell \leq F_n$. In $\ell$-$R_f(F_{n+1}-1)$, any vertex $v_j$ that forms an arc with $v_i$ is given by $$  v_i \rightarrow 
\begin{cases}
v_{i+(k-\ell)} & \text{for} 1\leq i \leq F_{n+1}-(k-\ell ) \\
v_{i+(k-\ell)-F_{n+1}} & \text{for} F_{n+1}-(k-\ell)+1 \leq i \leq F_{n+1} \\
v_{i+(k-\ell)-F_n} & \text{for} F_n-(k-\ell)+1 \leq  i \leq F_n
\end{cases} $$
The total number of arcs listed above are $(F_{n+1}-(k-\ell))+(k-\ell)+(k-\ell) = F_{n+1}+(k-\ell)=2k-\ell+1$. The indegree and outdegree of any vertex can be known from the Figure \ref{fig:4}. 

For gievn $k=F_{n+1}-1$, $k-\ell <F_{n-1}$ and $2(k-\ell)<F_n$, there exist no vertex $v_i$ in the graph $\ell$-$R_f(F_{n+1}-1)$ such that $deg_{in}(v_i)=2=deg_{out}(v_i)$.
\begin{figure}
    \centering
    \includegraphics[scale=0.25]{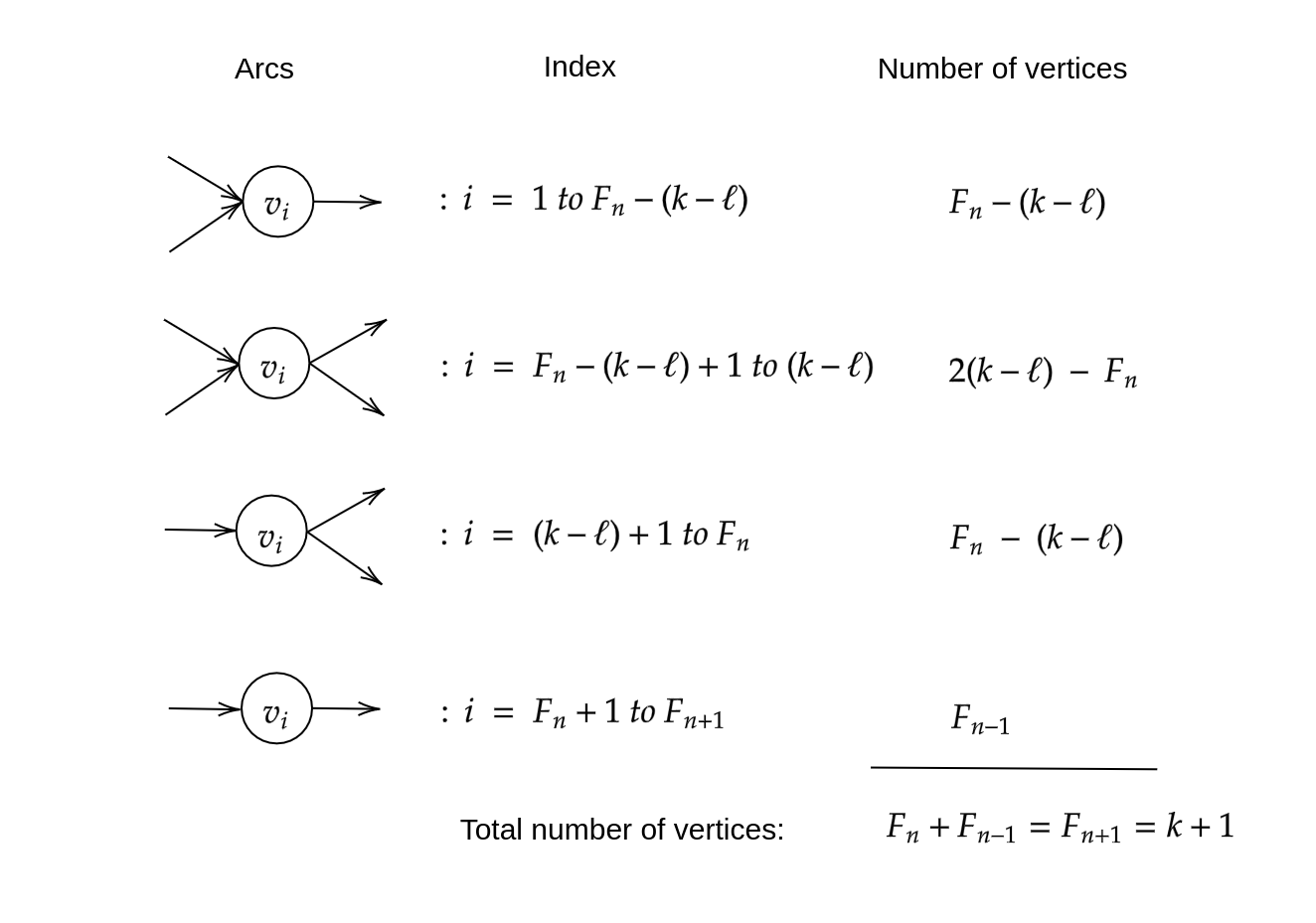}
    \caption{Indegree and out degree of any vertex $v_i$}
    \label{fig:4}
\end{figure}

   The $\ell$-Rauzy graph of order $k$ for the infinite Fibonacci word $f$ is isomorphic to the $\ell$-Rauzy graph of order $k$ for the complement of infinite Fibonacci word $f^c$. It is proved in the following proposition.
 \begin{proposition}
  Let $w=f $ and $w'=f^c$. Then $\ell$-$\mathbbmss{R}_w(k)\simeq \ell$-$ \mathbbmss{R}_{w'}(k),~\forall ~k\in \mathbbmss{N}$, $1\leq \ell \leq k-1$. 
\end{proposition}
  \begin{proof}
  If $x\in V(\ell$-$ \mathbbmss{R}_w(k))$ then $x^c\in \ell$-$\mathbbmss{R}_{w'}(k) $. A morphism $\phi : \ell$-$ \mathbbmss{R}_w(k)\rightarrow \ell$-$\mathbbmss{R}_{w'}(k)$ is given by $\phi (x)=x^c$, where $x\in V(\ell$-$\mathbbmss{R}_w(k))$. Also, the arcs $(u,v)\in E(\ell$-$\mathbbmss{R}_w(k))\Leftrightarrow (u^c,v^c)\in E(\ell$-$\mathbbmss{R}_{w'}(k))$. Hence, $\phi$ is an isomorphism and $\ell$-$\mathbbmss{R}_w(k)\simeq \ell$-$\mathbbmss{R}_{w'}(k)~\forall ~k\in \mathbbmss{N}$.
  \end{proof}
  Any two $\ell$-Rauzy graphs for the infinite Fibonacci word are not isomorphic to each other, is proved in the following theorem.
\begin{theorem}
The $\ell$-$\mathbbmss{R}_f(k_1) $ is not isomorphic to $\ell'$-$\mathbbmss{R}_f(k_2)$ for any $k_1\neq k_2$ or $\ell \neq \ell'$.
\end{theorem}
\begin{proof}
The $\ell$-Rauzy graph of infinite Fibonacci word $\ell$-$\mathbbmss{R}_f(k_1) $ has $|V_1|=k_1+1$ and $|E_1|=2k_1-\ell+1$ where as $\ell'$-$\mathbbmss{R}_f(k_2) $ has $|V_2|=k_2+1$ and $|E_2|=2k_2-\ell'+1$. 

In the case $1$: $k_1\neq k_2$, as the cardinality of vertex set of $\ell$-$\mathbbmss{R}_f(k_1) $ is different from $\ell'$-$\mathbbmss{R}_f(k_2) $, they are not isomorphic graphs.

In the case $2$: $k_1=k_2$, the cardinality of arc set of $\ell$-$\mathbbmss{R}_f(k_1) $ is different from $\ell'$-$\mathbbmss{R}_f(k_2) $, and so they are not isomorphic graphs.\end{proof}
We show that there exist a non-trivial bijection between the $\ell$-Rauzy graph and Rauzy graph of order $k$ for the infinite Fibonacci word $f$, but not an isomorphism.
\begin{theorem}
There exist a mapping $\psi: \ell$-$\mathbbmss{R}_f(k)\to \mathbbmss{R}_f(k) $ such that $\psi $ is a bijection.
\end{theorem}
\begin{proof}
Let $\psi: \ell$-$\mathbbmss{R}_f(k)\to \mathbbmss{R}_f(k) $ be a mapping. By definition, $V(\ell$-$\mathbbmss{R}_f(k))=V(\mathbbmss{R}_f(k))=F(k)$. Each arc $e\in E(\ell$-$\mathbbmss{R}_f(k)) $ is a word of length $2k-\ell$ and each path $v_iv_{i+1}\cdots v_{i+k-\ell}$ or $e'_ie'_{i+1}\cdots e'_{(i-1)+k-\ell}$ in $\mathbbmss{R}_f(k)$ is a word of length $(k+\underbrace{1+1+\cdots +1}_{(k-\ell)~times})=2k-\ell$. Now, we map each arc $e=v_iv_{i+ k-\ell}$ in $\ell$-$\mathbbmss{R}_f(k)$ to the path $v_iv_{i+1}\cdots v_{(i-1)+k-\ell}v_{i+k-\ell }$ or $e'_ie'_{i+1}\cdots e'_{(i-1)+k-\ell}$. The mapping $\psi$ is a bijection because $E(\ell$-$\mathbbmss{R}_f(k))=\{P_f(k-\ell)\}$, where $P_f(k-\ell)$ is the path of length $(k-\ell)$ in infinite Fibonacci word. 
\end{proof}
In the above theorem, $\psi$ becomes an isomorphism only if it is a bijection mapping between the arc sets of $\ell$-$\mathbbmss{R}_f(k)$ and $\mathbbmss{R}_f(k)$. Here, we have given a bijection between the arc set of $\ell$-$\mathbbmss{R}_f(k)$ and the $\{P_f(\frac{k}{2})\}$ i.e., the set of all paths of length $(k-\ell)$ in $\mathbbmss{R}_f(k)$.

\section{Main result}\label{sec5}
 In this section, we prove that the $\ell$-Rauzy graph of order $k$ for the infinite Fibonacci word is strongly connected for any $k,\ell\in \mathbbmss{N},~1\leq \ell \leq k-1$.
\begin{theorem}\label{thm3}
For a given $k>1,~1\leq \ell \leq k-1$ and$~k,~\ell \in \mathbbmss{N}$, the $\ell$-Rauzy graph of infinite Fibonacci word $f$ of order $k$, i.e., $\ell$-$\mathbbmss{R}_f(k)$ is strongly connected.
\end{theorem}
\begin{proof}
For a given $k>1,~1\leq \ell \leq k-1$ and $~k,~\ell \in \mathbbmss{N}$, the distinct subwords of length $k$ in infinite Fibonacci word $f$ is the set of all vertices in $\ell$-$\mathbbmss{R}_f(k)$. It is well known that the number of subwords of Fibonacci infinite word of length $k$ is $k+1$. Let the vertices of $\ell$-$\mathbbmss{R}_f(k)$ be $v_1,~v_2,\cdots,~v_k,~v_{k+1}$.

For a given $k$, $F_n\leq k \leq F_{n+1}$. From proposition $2.7$ in \cite{chuan}, the first occurrences of $k+1$ distinct factors of length $k$ are given by 
$$v_j= \begin{cases}
f[j;~k] & \text{if~~}1\leq j \leq F_n \\
f[j+F_{n+1}-(k+1);~k] & \text{if~~}  F_n+1\leq j\leq k+1
\end{cases} $$
From corollary $3.6$ and proposition $3.9$ in \cite{chuan}, all the locations of $v_j$ are given by 
$$loc.(v_j)=\begin{cases}
\{tF_{n-1}+\lfloor(t+1)\tau\rfloor F_{n-2}+j \} & \text{if~~}1\leq j\leq F_{n+1}-k-1 \\
\{tF_n+\lfloor(t+1)\tau\rfloor F_{n-1}+j \} & \text{if~~} F_{n+1}-k\leq j\leq F_n\\
\{tF_{n+1}+\lfloor(t+1)\tau \rfloor F_n+j+F_{n+1}-(k+1) \} & \text{if~~} F_{n+1}\leq j\leq k+1
\end{cases}$$
where $t\geq 0$ in each of those sets. We see that locations of $v_j$ for any $j$ is of the form 
$$bt+c\lfloor(t+1)\tau\rfloor +d $$
where $b,~c\in \{F_{n-2},~F_{n-1},~F_n,~F_{n+1} \}$ and $d=j$ or $j+F_{n+1}-(k+1)$.

Let us consider the path $(say~P_1)$ that starts from the subword of length $k$, located in the first position of infinite Fibonacci word. By the definition of $\ell$-Rauzy graphs, the path $P_1$ is given by $$ f[1;k] \rightarrow f[1+(k-\ell);k]\rightarrow f[1+2(k-\ell);k]\rightarrow \cdots \rightarrow f[1+m(k-\ell);k]\rightarrow \cdots $$
In path $P_1$, it is clear that any subword of the form $f[1+m(k-\ell);k]$ is reachable from $f[1;k]$ or $v_1$. If atleast one location of each vertex is of the form $1+m(k-\ell)$, then every vertex is reachable from $v_1$.

The integer solutions to the equation $$ 1+m(k-\ell) = bt+c\lfloor(t+1)\tau\rfloor +d  \text{~~~~~~for each $1\leq j \leq k+1$} $$ 
guarantee that atleast one location of each vertex is of the form $1+m(k-\ell)$. Let $x_1=m,~x_2=t,~x_3=\lfloor (t+1)\tau \rfloor$ be the variables. The equation can be rewritten as $$ ax_1-bx_2-cx_3=d'. $$
For each $1\leq j\leq k+1$, the linear Diophantine equation $ ax_1-bx_2-cx_3=d' $ has infinite integer solutions $\iff$ $g.c.d(a,b,c)|d'$.

It is well known that any two consecutive Fibonacci numbers are coprime, $g.c.d(b,c)=1$ for any $1\leq j\leq k+1$, and so $g.c.d(a,b,c)=1$ that divides $d'$ always.

Now, it is clear that the equation $ax_1-bx_2-cx_3=d'$ has infinite integer solutions for any $1\leq j\leq k+1$. Thus, every vertex is reachable from $v_1$ in the path $P_1$. As every vertex is located infinitely many times in the path $P_1$, the vertex $v_1$ is reachable from any other vertex. Hence, $\ell$-$\mathbbmss{R}_f(k)$ is strongly connected.
\end{proof}
However, the $\ell$-Rauzy graph of order $k$ for any recurrent word need not be connected. Figure \ref{fig:3} shows that the $2$-Rauzy graph of order $4$ for the Thue-Morse word (aperiodic recurrent infinite word) is not connected.

\end{document}